\documentclass{amsart}

\usepackage{amssymb, amsmath,textcomp}
\usepackage{hyperref}
\usepackage{srcltx}
\usepackage{color}

\renewcommand{\marginpar}[1]{}

\numberwithin{equation}{section}

\theoremstyle{plain}
\newtheorem{proposition}{Proposition}[section]
\newtheorem{theorem}[proposition]{Theorem}
\newtheorem{lemma}[proposition]{Lemma}
\newtheorem*{claim}{Claim}
\newtheorem*{thm}{Theorem}

\theoremstyle{definition} 
\newtheorem*{remark}{Remark}
\newtheorem{fact}[proposition]{Fact}
\newtheorem{definition}[proposition]{Definition}

\newtheorem*{question}{Question}

\newcommand{\Q}{\mathbb{Q}}
\newcommand{\C}{\mathbb{C}}
\newcommand{\N}{\mathbb{N}}
\newcommand{\Z}{\mathbb{Z}}

\newcommand{\tr}{\operatorname{trdeg}}
\newcommand{\rk}{\operatorname{rk}}
\newcommand{\ldim}{\operatorname{ldim}}
\newcommand{\ex}{\operatorname{exp}}
\newcommand{\dd}{\delta}
\newcommand{\defh}[1]{\langle #1 \rangle}
\newcommand{\Spec}{\operatorname{Spec}}
\newcommand{\alg}{\operatorname{alg}}

\def \<{\langle}
\def \>{\rangle}

\begin{document}

\bibliographystyle{plain}

\title{Generic points on exponential curves}
\date{\today}

\author{Ayhan G\"unaydin and Amador Martin-Pizarro}
\address{Centro de Matem\'atica e Aplica\c{c}oes Fundamentais (CMAF), Avenida Professor Gama Pinto, 2, 1649-003 Lisbon, Portugal }
\email{ayhan@ptmat.fc.ul.pt}

\address{Universit\'e de Lyon; CNRS; Universit\'e Lyon 1; INSA de Lyon F-69621; Ecole Centrale de Lyon; Institut Camille Jordan UMR5208, 43 boulevard du 11 novembre 1918, F--69622 Villeurbanne Cedex, France }

\email{pizarro@math.univ-lyon1.fr}

\thanks{The first author was supported by the FCT grant SFRH/BPD/47661/2008. The second author conducted research 
 within the French Research Network {\bf MODIG} ANR-09-BLAN-0047 as well as an Alexander von Humboldt Stiftung
Forschungsstipendium f\"ur erfahrene Wissenschaftler  1137360}


\begin{abstract} 
We show, assuming Schanuel's conjecture, that every irreducible complex polynomial in two variables
 where both variables appear has infinitely many algebraically independent solutions of the form $(z,e^z)$.  
\end{abstract}
\maketitle

\section*{Introduction}
A long-standing relevant conjecture in transcendence theory is Schanuel's conjecture,
which states that given $\Q$-linearly independent complex numbers $a_1,\ldots,a_n$, we have

$$\tr_\Q (a_1,\ldots,a_n, \ex (a_1),\ldots,\ex(a_n))\geq n.$$

\noindent
This conjecture implies formally  Lindemann-Weierstrass Theorem as well as the algebraic independence of $e$ and $\pi$. 
In \cite{Zilber-Pseudoexp}, Zilber noticed that Schanuel's conjecture could be interpreted in the setting of 
a predimension function \`a la Hrushovski. Put

$$\dd(A) := \tr_\Q (A, \ex(A)) - \ldim_\Q A, $$

\noindent
for any finite subset $A$ of $\C$. Then Schanuel's conjecture is equivalent to $\emptyset$ 
being \emph{strong} 
(or \emph{self-sufficient}) in $\C$, that is, $\dd(A)\geq 0$ for any finite dimensional $\Q$-subspace $A$. 
Zilber considered the language $L$ of rings augmented by a unary function 
symbol $\ex$ and $L$-structures (called {\it pseudo-exponential fields}) consisting of algebraically closed fields $F$
of characteristic $0$ equipped with a surjective group homomorphism $\exp$ from its additive group onto its multiplicative 
group whose kernel is generated by a single transcendental element. 

\noindent The above predimension function induces a finitary 
pregeometry in which given a finite subset $A$ of $F$, its $\delta$-closure consists of all the elements $b$ in 
$F$ such that for some finite dimensional $\Q$-subspace $B$ containing $A\cup\{b\}$,
$$ \dd(B/A):= \dd(B)-\dd(A) \leq 0.$$

\noindent Moreover, Zilber's fields $F$ satisfy the 
following extra properties:
\begin{enumerate}
 \item[(SC)] $\dd(A)\geq 0$ for any finite $A\subseteq F$.
 \item[(CCP)] The $\delta$-closure of any finite set is countable.
 \item[(EC)] Given a variety $V \subset F^{2n}$ defining a minimal extension of predimension $0$, 
there is a generic point in $V$ of the form $(z,\ex(z))$. Equivalently, there are infinitely many 
algebraically independent such points in $V$.
\end{enumerate}

Though these conditions are not expressible in first-order logic, the class $\mathcal C$ of pseudo-exponential 
fields is axiomatizable in $\mathcal{L}_{\omega_1,\omega}(Q)$, where $Q$ denotes the quantifier 
\emph{``there exists uncountably many''}. Zilber showed that  $\mathcal{C}$ is  uncountably categorical, 
that is, there is a unique such field in each uncountable cardinal (up to isomorphism). Moreover, 
definable sets are either countable or co-countable. The question then is whether $\C$ is the unique 
such field of cardinality $2^\omega$. By a clever use of Ax's theorem \cite{Ax} on Schanuel's condition for 
the field of Laurent series in one variable, he concluded that $\C$ already satisfies condition CCP. Note that 
since Schanuel's conjecture is part of the axioms, the natural question is then the following:

\begin{question}
Assume $\C$ satisfies SC. Does it follow that $\C$ satisfies EC? 
\end{question}

\noindent
A warm-up case is when the variety is a curve in $\C^2$ given by an irreducible complex polynomial $p$
in two variables where both variables appear. Let $f$ be the entire function given by $f(z)=p(z,\ex(z))$.
In \cite{marker_exp}, Marker  proved that such a function has infinitely many algebraically independent zeros 
if $p$ is in $\Q^{\alg}[X,Y]$.

\medskip\noindent
In this article, we extend Marker's result to all complex polynomials.

\begin{thm}\label{T:main}
Suppose Schanuel's conjecture is true. Then for an irreducible complex polynomial $p$ in two variables 
where both variables appear, the entire  function $f(z):=p(z,\ex(z))$ has infinitely many algebraically independent zeros.
\end{thm}

\medskip\noindent
The authors would like to express their sincere gratitude to Angus Macintyre and Dave Marker for their stimulating 
comments and help during a previous version of this work. 

\section{Linear relations in fields of finite transcendence degree}
All throughout this section, let $K$ be an algebraically closed subfield of $\C$ of 
finite transcendence degree $d$ containing $\pi$. Put $\Gamma:=\ex(K)$, a subgroup of $\C^\times$. We consider 
solutions in $\Gamma$ of 
\begin{equation*}\label{mainEQ}\tag{\textborn}
 \lambda_1 x_1+\cdots+\lambda_k x_k=1,
\end{equation*}
where $\lambda_1,\ldots,\lambda_k\in K$. We say that a solution $\vec\gamma=(\gamma_1,\dots,\gamma_k)$ in $\Gamma$ 
of (\ref{mainEQ}) is {\it non-degenerate} if $\sum\limits_{i\in I} \lambda_i \gamma_i\neq 0$ for every nonempty proper 
subset $I$ of $\{1,\ldots,k\}$. 

\medskip\noindent
We begin with some notations that will be useful in the rest of the paper.

\begin{definition}
Let $G$ be an abelian group, written multiplicatively and for $n>0$ put $G^{[n]}=\{g^n:g\in G\}$. We say that a 
subgroup $H$ is \emph{pure} in $G$ if $H\cap G^{[n]}=H^{[n]}$ for all $n>0$.  We say that $H$ is \emph{radical} in 
$G$ if it is pure in $G$ and it contains all the torsion elements of $G$.

\medskip\noindent
Given $A\subseteq G$, we set $\defh{A}_G$ to be the smallest radical subgroup of $G$ containing $A$. That is,
$$ \defh{A}_G =\{ g\in G\, |\,g^n \in [A]_G \text{ for some } n\in\N\}$$
where $[A]_G$ is the subgroup generated by $A$.  When $G$ is clear from the context, we will drop the subscripts 
and just write $\defh{A}$ and $[A]$. For instance throughout this section the ambient group is  $\C^\times$ 
unless explicitly stated otherwise.
\end{definition}

\medskip\noindent
Since $K$ is a field of characteristic $0$, it is easy to see that $\Gamma$ is divisible; 
in particular it is pure in $\C^\times$. Moreover $\Gamma$ is a radical subgroup of $\C^\times$
 since $\sqrt{-1}\pi$ is in $K$. Therefore, so is  $\Gamma\cap K^\times$.  

\medskip\noindent
Given $a_1,\dots,a_n$ in $\C$, by $\vec a$ we denote the tuple $(a_1,\ldots,a_n)$ and $\ex(\vec a)$ 
denotes $(\ex(a_1),\dots,\ex(a_n))$.  

\medskip\noindent
Note the following straight-forward consequence of Schanuel's Conjecture concerning the rank of $\Gamma\cap K^\times$.

\begin{lemma}\label{L:rk}{(Assuming Schanuel's conjecture)}\\
\indent The rank of $\Gamma\cap K^\times$ is at most $d$.
\end{lemma}

\begin{proof}
 Let $\beta_1,\ldots,\beta_{d+1}$ in $K$ such that $\ex(\vec\beta)$ is in $\Gamma\cap K^\times$. In particular, 
$$ \tr_\Q (\beta_1,\ldots,\beta_{d+1},\ex(\beta_1),\ldots,\ex(\beta_{d+1}))\leq d.$$
Thus Schanuel's conjecture implies that $\beta_1,\ldots,\beta_{d+1}$ are $\Q$-linearly dependent and 
hence $\ex(\beta_1),\ldots,\ex(\beta_{d+1})$ are multiplicative dependent.
\end{proof}

\noindent
On the basis of this lemma, take $\beta_1,\dots,\beta_t\in K$ where $t\leq d$ such that 
$\pi\sqrt{-1},\beta_1,\dots,$ $\beta_t$ are $\Q$-linearly independent and
$$\Gamma\cap K^\times=\<\ex(\beta_1),\dots,\ex(\beta_t)\>.$$

\noindent Recall Lemma 8.2 from \cite{vdDG}.

\begin{lemma}\label{T:Louayhan}
Let $F$ be a field with a subfield $E$ and subgroups $G,H$ of $F^\times$. Suppose also that $H$ is a 
radical subgroup of $G$. Then the following two conditions are equivalent:
\begin{enumerate}
 \item for every $\lambda_1,\ldots,\lambda_k\in E$, the equation (\ref{mainEQ}) has the same 
non-degenerate solutions in $H$ as in $G$.
\item whenever $g_1,\ldots,g_n$ in $G$ are multiplicatively independent over $H$, they are algebraically 
independent over the field $E(H)$. 
\end{enumerate}
\end{lemma}

\medskip\noindent
This allows us to prove the following.

\begin{proposition}\label{P:fterad}{(Assuming Schanuel's conjecture)}\\
\indent There exists a radical subgroup $\Gamma^*$ of $\Gamma$ of finite rank containing $\Gamma\cap K^\times$ 
such that for every $\lambda_1,\ldots,\lambda_k$ in $K$, the equation (\ref{mainEQ}) has the same non-degenerate 
solutions in $\Gamma^*$ as in $\Gamma$.
\end{proposition}

\begin{proof}

Fix a transcendence basis $\{\alpha_1,\dots,\alpha_s\}$ of $K$ over $\Q(\beta_1,\dots,\beta_t)$ and let 
 $e=\tr_\Q K(\ex(\vec\alpha))$. In particular, $\alpha_1,\dots,\alpha_s,\beta_1,\dots,\beta_t$ are 
$\Q$-linearly independent and  $d\leq e\leq d+s$.  

\smallskip\noindent
By the previous lemma, we only need to  construct a radical subgroup $\Gamma^*$ of $\Gamma$ of finite rank with the 
following property:

\medskip\noindent
{($*$)}  Let $\gamma_1,\dots,\gamma_m\in\Gamma$ be algebraically dependent over $K(\Gamma^*)$. Then they are 
multiplicatively dependent over $\Gamma^*$.

\medskip\noindent
We construct inductively an increasing chain of radical subgroups $\{\Gamma_{i}\}_{i\in\N}$ of $\Gamma$ of finite rank. Start 
with  
$$\Gamma_0:=\defh{(\Gamma\cap K^\times) \cup \ex(\vec \alpha)}.$$ 
For $i>0$, if there exist $b_i^{(1)},\dots,b_{i}^{(m_i)}$ in $K$ such that 
$\exp(\vec{b_{i}})$ is multiplicatively independent over 
$\Gamma_{i-1}$ but algebraically dependent over $K(\Gamma_{i-1})$, then let 
$\Gamma_i:=\<\Gamma_{i-1}\cup\{\exp(\vec{b_{i}})\}\>$; otherwise let $\Gamma_{i}:=\Gamma_{i-1}$. 
(Note that the construction of the chain is not unique as it depends on the choices of the $\vec{b_i}$'s.)

\medskip\noindent
We only need to show that the chain stabilises after $i=d-t$. Since at every step we add a tuple which is algebraically 
dependent over the previous one, we have that the transcendence degree of 
$K(\Gamma_{i})$ is at most $e+(\rk(\Gamma_{i})-s-t)-i$. In particular, 

$$\tr_\Q K(\Gamma_{d-t})\leq\rk (\Gamma_{d-t}).$$

Suppose there is some $\vec d$ in $ K^n$ such that $\exp(\vec d)$ is algebraically dependent over $K(\Gamma_{d-t})$.
 Then the transcendence degree
$$\tr_{\Q}(\vec\alpha,\vec\beta,\{\vec{b_i}\}_{1\leq i\leq d-t} ,\vec d,\exp(\vec\alpha),\exp(\vec\beta),
\{\exp(\vec{b_i})\}_{1\leq i\leq d-t},\exp(\vec d))$$ is at most $\tr_\Q(K(\Gamma_{d-t}))+n-1$ and hence strictly
 less than $\rk(\Gamma_{d-t})+n$. Schanuel's conjecture yields a $\Q$-dependence among 
$\vec\alpha,\vec\beta,\{\vec{b_i}\}_{1\leq i\leq d-t} ,\vec d$. By the construction of the chain, $\exp(\vec d)$ is multiplicatively dependent over 
$\Gamma_{d-t}$, showing that $\Gamma_{d-t}=\Gamma_{i}$ for $i\geq d-t$ and that ($*$) is satisfied with $\Gamma^*:=\Gamma_{d-t}$.
 
\end{proof}

\begin{remark} It follows from the proof above that if the rank of $\Gamma\cap K^\times$ is already $d$, then we can take 
$\Gamma^*$ to be $\Gamma\cap K^\times$.
\end{remark}

Let $\Gamma^*=\<\ex(a_1),\dots,\ex(a_s)\>$ with $a_1,\dots,a_s\in K$ linearly independent over $\sqrt{-1}\pi$. 

\medskip\noindent From now on, $\mathbb U$ denotes the multiplicative group of all roots of unity. Recall the 
following results. 

\begin{lemma}\label{extension}{(Lemma 6.1 in \cite{MP})} \\
\indent Let $E\subseteq F$ be fields such that $E\cap\mathbb U=F\cap\mathbb U$ and $G$ be a radical subgroup of $E^\times$.
Then for $\lambda_1,\dots,\lambda_n\in E^\times$, the equation (\ref{mainEQ}) 
has the same non-degenerate solutions in $G$ as in $\defh{G}_{F^\times}$.
\end{lemma}

\begin{lemma}\label{P:BaysZil}{(Proposition 2.2 (ii) in \cite{Bays-Zilber})} \\
Let $L$ be a finitely generated extension of $\Q(\mathbb U)$. Then the quotient group $L^\times/\mathbb U$ is a free abelian group.
\end{lemma}

\begin{remark}
In \cite{MP}, the statement of Lemma \ref{extension} is not quite correct.  There $G$ is taken to be a pure subgroup of $E^\times$, which is not enough to get the conclusion. However, it is easy to see that the proof there proves the lemma as stated here.
\end{remark}

We can now reduce our situation from $K$ to any subfield $L$  that is finitely generated over $\Q(\mathbb U)$ containing
the generators $\exp(\vec a)$. 
\begin{lemma}\label{finitelygeneratedoverU}
 Let $L$ be a finitely generated extension of $\Q(\mathbb U)$ containing $\exp(\vec a)$.  
Then there are $c_1,\dots,c_{t'}$ in $K$ linearly independent over $\sqrt{-1}\pi$ such that for every 
$\lambda_1,\dots,\lambda_k\in L$, all the nondegenerate 
solutions of (\ref{mainEQ}) in $\Gamma^*$ are in $\mathbb U\cdot[\exp(\vec c)]$.
\end{lemma}

\begin{proof}
In Lemma \ref{extension}, take $L$, $\C$ and $\defh{\exp(\vec{a})}_{L^\times}$ in the place of  $E,F$ and $G$.  
Note that then $\defh{G}_{F^\times}$ is nothing other than $\Gamma^*$.  

\medskip\noindent
Being a divisible group, $\mathbb U$ splits in $\defh{\exp(\vec{a})}_{L^\times}$ and consider its complement $G'$.
Lemma \ref{P:BaysZil} yields that $G'$ is a subgroup of finite rank of a free 
commutative group, therefore   $G'$ is finitely generated, say 
$G'=[\exp(\vec c)]$, with $c_1,\dots,c_{t'}\in K$. We may clearly assume that $\vec c$ is  
linearly independent over $\sqrt{-1}\pi$.

Therefore the possible nondegenerate solutions in $\Gamma^*$ 
of (\ref{mainEQ}) with  $\lambda_1,\dots,\lambda_k\in L$ are of the form
$$(\zeta_1\exp(\vec{m}_1\cdot \vec{c}),\dots,\zeta_k\exp(\vec{m}_k\cdot \vec{c}))$$
for some integer tuples $\vec{m}_1,\dots,\vec{m}_k$ and roots of unity $\zeta_1,\dots,\zeta_k$.
\end{proof}

\section{Specializations and reduction to a number field}

We first remark the following easy observation, whose proof follows the lines of the proof of Lemme 4 of \cite{LaurentII}
(Note that Laurent considered only finitely generated 
$\Q$-algebras, however his result is deeper).

\begin{lemma}\label{LinInd_spe}
Let $R$ be a subring of $\bar\Q[S]$, where $S$ is a finite subset of $\C$. Suppose that $b_1,\dots,b_q$ 
are elements of $R$ and let $q'$ be the linear dimension over $\bar\Q$ of $\vec b$.  
Then there are ring homomorphisms $\phi_1,\dots,\phi_{q'}$ from $R$ to $\bar\Q$ fixing $k:=R\cap\bar\Q$ such 
that for every $\alpha_1,\dots\alpha_{q}$ in $k$ with $\alpha_1 b_1+\cdots+\alpha_q b_q\neq 0$ there is
 some $i\in\{1,\dots,q'\}$  with $\phi_i(\alpha_1 b_1+\cdots+\alpha_q b_q)\neq 0$.
\end{lemma}

\begin{proof}
 After changing $R$, we may assume that $q=q'$. It suffices then to find  $\phi_i:R\to \bar\Q$ for each 
 $i\in\{1,\dots,q\}$ fixing $k$ such that the determinant
$$D_q:=\left| \begin{array}{ccc} \phi_1(b_1) & \cdots & \phi_1(b_q) \\ \vdots & \  & \vdots \\ \phi_q(b_1) & \cdots 
& \phi_q(b_q)\end{array} \right| $$
is nonzero.

\medskip\noindent
We proceed by induction on $q$. For $q=1$ by Nullstellensatz take a ring homomorphism $R[b_1^{-1}]\to \bar\Q$ that fixes 
$k$. Clearly, its restriction to $R$ sends $b_1$ to some non-zero element. 

\medskip\noindent
Assume now that  $\phi_1,\dots,\phi_{q-1}$ have been already constructed such that $D_{q-1}\neq 0$. Then the determinant
$$D'_q:=\left| \begin{array}{ccc} \phi_1(b_1) & \cdots & \phi_1(b_q) \\ \vdots & \  & \vdots \\ \phi_{q-1}(b_1) & \cdots 
& \phi_{q-1}(b_q) \\  b_1 & \cdots & b_q\end{array} \right| $$
is $\beta_1 b_1+\cdots+\beta_q b_q$, where $\beta_1,\dots,\beta_q$ are algebraic numbers. In particular, by induction,
 $\beta_q=D_{q-1}\neq 0$. Therefore, since we are assuming that the tuple $\vec b$ is $\bar \Q$-linearly independent, we conclude 
that $D'_q\neq 0$. Consider $$R':=R[ (D'_q)^{-1}].$$
Nullstellensatz implies that there is a ring morphism $\phi_q$ from $R'$ to $\bar\Q$ fixing $k':=R'\cap\bar\Q$. 
Its restriction to $R$ has the property that $\phi_q D'_q\neq 0$ which implies that $D_q\neq 0$.
\end{proof}

In order to reduce our setting to a number field in the last section, we need to carefully choose a specialization 
to $\bar\Q$. This is ensured by the density of closed points in specific subsets of the spectrum of 
any finitely generated $\Q$-algebra $R$. Given such $R$ and a polynomial $Q$ over $R$ irreducible in 
$\mathrm{Frac}(R)[X]$, denote by $\Omega(Q)$ the collection of prime ideals $\mathfrak{p}$ of $R$ such 
$Q$ mod $\mathfrak{p}$ has the same degree as $Q$ and it is irreducible as a polynomial over
 $\mathrm{Frac}(R/\mathfrak{p})$. Recall that a \emph{Hilbert set} $\Omega$ is a subset of $\Spec(R)$ which
 contains a finite intersection of non-empty open sets and sets of the form $\Omega(Q)$. 

\begin{fact}\label{F:Hilbert} Let $R$ be a finitely generated $\Q$-algebra.
\begin{itemize} 
 \item[(i)] Given a finitely generated subgroup $G$ of $R^\times$, there is a Hilbert set
 $\Omega$ such that the residue map $ G\to (R/\mathfrak{p})^\times$ is injective for every $\mathfrak{p}$ in $\Omega$. 
\item[(ii)] For any Hilbert set $\Omega$ in $R$, the collection of maximal ideals contained in $\Omega$ 
is dense in $\Spec(R)$. 
\end{itemize}
\end{fact}

\medskip\noindent
Combining the above with the proof of Lemma \ref{LinInd_spe}, one obtains the following result.

\begin{lemma}\label{lemme4laurent}{(Lemme 4 in \cite{LaurentII}\!)}\\
 Let $R$ be a finitely generated $\Q$-algebra with largest subfield $k$ and $G$ a finitely generated subgroup of 
$R^\times$. Suppose also that $b_1,\dots,b_q$ are elements of $R$ that generate a $\bar\Q$-linear space of dimension 
$q'$. Then there are ring homomorphisms $\phi_1,\dots,\phi_{q'}$ from $R$ into $\bar\Q$ such that each $\phi_i$ is 
injective on $G$ and that for every $\alpha_1,\dots\alpha_{q}\in k$ with $\alpha_1 b_1+\cdots+\alpha_q b_q\neq 0$,
 there is $i\in\{1,\dots,q'\}$ with $$\phi_i(\alpha_1 b_1+\cdots+\alpha_q b_q)\neq 0.$$
\end{lemma}

In order to bound the degrees of the roots of unity appearing in Lemma \ref{finitelygeneratedoverU} we will need  
the following result.

\begin{theorem}\label{dvornicich-zannier}{(Theorem 1 in \cite{dvornicich-Zannier}\!)}\\
 Let $F$ be a number field, $a_0,a_1,\dots,a_k$ in $F$ and $\zeta$ a root of unity of order $Q$ such that 
$a_0+\sum\limits_{j=1}^{k}a_j\zeta^{n_j}=0$ with $\gcd(Q,n_1,\dots,n_k)=1$.  Let $\delta=[F\cap\Q(\zeta):\Q]$ 
and suppose that for any nonempty proper subset $I$ of $\{0,1,\dots,k\}$ the sum 
$\sum\limits_{j\in I}a_j\zeta^{n_j}\neq 0$. Then for each prime $p$ and $n>0$, if $p^{n+1}|Q$, 
then $p^{n}|2\delta$ and
 $$k\geq \dim_F(F+F\zeta^{n_1}+\cdots+F\zeta^{n_k})\geq 1+\sum_{p|Q,p^2\nmid  Q}[\frac{p-1}{\gcd(\delta,p-1)}-1].$$
In particular, the order $Q$ of $\zeta$ is bounded by a constant depending on $k$ and $\delta$ (and therefore $[F:\Q]$).
\end{theorem}

The last result of this section concerns work from \cite{Laurent}. Work inside a number field  $F$. 
For $t$, $r$ in $\N$ consider polynomials $Q_1,\dots,Q_r$ over $F$ in $t$ many variables as well as 
a finite set $Z:=\{a_{ji}\,:\,j=1,\dots,r\,;\,i=1,\dots,t\}$ in $F^\times$ . We are interested in describing the set of 
tuples  $\vec m$ in $\Z^t$ such that

\begin{equation*}\label{Eq:Laurent}\tag{**}
 \sum\limits_{j=1}^r Q_j(\vec m)\prod\limits_{i=1}^t a_{ji}^{m_i}=0.
\end{equation*}

\noindent For such an equation (\ref{Eq:Laurent}), let $H$ be the subgroup of those $\vec m$ in $\Z^t$  
such that $$\prod\limits_{i=1}^t a_{ji}^{m_i}=\prod\limits_{i=1}^t a_{j'i}^{m_i},$$ \noindent for every $j,j'\in\{1,\dots,r\}$.

\medskip\noindent
Th\'{e}or\`{e}me 6  of \cite{Laurent} describes precisely the solutions of  (\ref{Eq:Laurent}), however for our purposes
 the following simplified version suffices. 

\begin{theorem}\label{Laurent}
Suppose that $H$ is trivial. Then there are constants $\delta$, $\eta$  depending only on $Z$ and the field $F$ 
such that if $\vec m$ in $\Z^t$ satisfies (\ref{Eq:Laurent}) and for every nonempty proper $J\subseteq\{1,\dots,r\}$ 
the sum $\sum\limits_{j\in J} Q_j(\vec m)\prod_{i=1}^t a_{ji}^{m_i}$ is nonzero, then 
$$||\vec m||\leq\delta \log||\vec m||+\eta,$$
\noindent where $||\vec m||:=\max_{i}{|m_i|}$.
\end{theorem}

\begin{remark}\label{R:boundm} The independence of the constants $\delta$, $\eta$ from the coefficients of $Q_i$ follows from the 
proof of \cite{Laurent}.  Therefore, there is some $N\in \N$ such that if $\vec m$ satisfies a non-trivial equation 
(\ref{Eq:Laurent}), then $||\vec m||\leq N$.
\end{remark}

\section{The Main Theorem}

\noindent
Here we prove the theorem stated in the introduction.

\medskip\noindent
We keep the notations from the previous sections. In particular, $K$ is an algebraically 
closed subfield of $\C$ of finite transcendence degree containing $\pi$ and the coefficients of $p$, which is an irreducible polynomial in two variables in which both variables appear.  

\medskip\noindent
Using Hadamard Factorization Theorem (see for instance \cite{Hormander}) and a result proved independently by Henson 
and Rubel \cite{Henson-Rubel} and by van den Dries \cite{vdDries-exp}, we have that $f(z)=p(z,\exp(z))$ has infinitely
 many zeros in $\C$ (for a proof of this, see \cite{marker_exp}).  Therefore in order to prove our theorem, it suffices 
to prove the following.

\begin{theorem}
 The entire function $f(z):=p(z,\exp(z))$ has finitely many zeros in $K$.
\end{theorem}

\begin{proof}
Write 
$$p(X,Y)=\sum_{j=0}^{m}p_j(X)Y^j,$$
where $p_j(X)\in K[X]$.
Also set $I=\{ j \in \{0,\ldots,m\}\, |\, p_j\neq 0\}$. Since $p$ is irreducible, $0$ lies in $I$. The set
 $\{z \in \C \,|\, p_j(z)=0 $ for some $j\in I\}$ is finite. Hence in order to show that there are finitely many 
solutions in $K$ to $p(z,\exp(z))=0$ we need only prove that 

$$ W :=\{ z \in K \,|\, p(z,\ex(z))=0 \wedge \bigwedge\limits_{j\in I} p_j(z)\neq 0\} $$
is finite.

\medskip\noindent
Let $z$ be in $W$. By considering the appropriate subsum, we may assume that $(\exp(z)^j)_{j\in I\setminus\{0\}}$ 
is a nondegenerate solution of 
$$\sum_{j\in I\setminus\{0\}}-\frac{p_j(z)}{p_0(z)}x_j=1.$$

Then by taking $L$ to be an appropriate finitely generated extension of $\Q(\mathbb U)$ in Lemma \ref{finitelygeneratedoverU}, there are 
$c_1,\dots,c_{t'}$ in $K$ such that $\exp(z)$ lies in the group 
$$\mathbb U\cdot[\exp(c_1),\dots,\exp(c_{t'})].$$
Then $$z\in\Q\pi\sqrt{-1}+\Z c_1+\cdots +\Z c_{t'}.$$

\medskip\noindent
We now apply Theorem \ref{dvornicich-zannier} to get a finer description of $W$. 

\begin{claim}
There is $N\in\N$ such that if $z\in W$ then there are  $k,l,m_1,\dots,m_{t'}$ in $\Z$ and $0<n<N$ such that $k<n$, $\gcd(k,n)=1$ and 
$$z=\frac{k2\pi\sqrt{-1}}{n}+l2\pi\sqrt{-1}+\sum\limits_{j=1}^{t'} m_j c_j$$
\end{claim}

\begin{proof}
Let $z\in W\setminus\{0\}$ and choose $k,l,m_1,\dots,m_{t'}\in\Z$ and $n>0$ such that $k<n$, $\gcd(k,n)=1$, and  
$$z=\frac{k2\pi\sqrt{-1}}{n}+l2\pi\sqrt{-1}+\sum\limits_{j=1}^{t'} m_j c_j.$$
We need to find a bound $N$ on $n$ independent of $k,l,m_1,\dots,m_{t'}$. 

\medskip\noindent
Set $\vec d=\ex(\vec c)$ and $\zeta=\exp(2\pi\sqrt{-1}/n)$. We then have 
 \begin{equation*}\label{MainEqIntegers}\tag{*}
  \sum_{j\in I}p_j(\frac{k2\pi\sqrt{-1}}{n}+l2\pi\sqrt{-1}+\sum\limits_{j=1}^{t'} m_j c_j)\zeta^{kj}\cdot(\vec d^{\vec m})^j=0.
 \end{equation*}

\medskip\noindent
Let $R$ be the $\bar \Q$-algebra generated by the coefficients of $p$, $\pi\sqrt{-1}$, $\vec c$ and $\vec d$ and their inverses. Using Lemma \ref{LinInd_spe} with appropriate $b_1,\dots,b_q$, choose by Lemma \ref{LinInd_spe} some specialization $\phi$ such that 
$$\phi(p_0(\frac{k2\pi\sqrt{-1}}{n}+l2\pi\sqrt{\-1}+\sum\limits_{j=1}^{t'} m_j c_j))\neq 0.$$

\medskip\noindent
The homomorphism $\phi$ transforms (\ref{MainEqIntegers}) into a  non-trivial relation 
$$\sum_{j\in I}\phi(p_j(\frac{k2\pi\sqrt{-1}}{n}+l2\pi\sqrt{-1}+\sum\limits_{j=1}^{t'} m_j c_j))\zeta^{kj}\cdot(\phi(\vec d)^{\vec m})^j=0.$$
Reorganizing we get a relation
$$\sum_{j\in I}a_j \zeta^{jk}=0,$$
where the $a_j$'s are algebraic numbers depending on $(n,k,l,\vec m)$ and not all zero. Note however that the number field $F$ containing the $a_j$'s is independent of $(n,k,l,\vec m)$. 

\medskip\noindent
Let $j_0=\gcd(n,j)_{j\in I\setminus\{0\}}$ and 
$\zeta_0=\exp(\frac{2\pi\sqrt{-1}}{n/j_0})$.  So we have a relation
$$\sum_{j\in I}a_j \zeta_0^{\frac{jk}{j_0}}=0.$$
For our purposes we may assume that no subsum is $0$. Then Theorem \ref{dvornicich-zannier} gives a bound on $n$
 depending only on the degree of $F$ and $|I|$.  Therefore there is $N>0$ depending only on $p(X,Y)$ such that if 
$(n,k,l,\vec m)$ satisfies (\ref{MainEqIntegers}), then $n<N$.  
\end{proof}

\medskip\noindent
Using this claim we may assume,  after modifying $f$ (finitely many times)  that its zeroes in $K$ are of the form 
$$l2\pi\sqrt{-1}+\sum\limits_{j=1}^{t'} m_j c_j$$ with $l,m_1,\dots,m_{t'}\in\Z$. Hence we have reduced the theorem to prove that there are only finitely many $(l,\vec m)\in\Z^{1+t'}$ such that
\begin{equation*}\label{MainEQ}\tag{***}
  \sum_{j\in I}p_j(l2\pi\sqrt{-1}+\sum\limits_{j=1}^{t'} m_j c_j )(\vec d^{\vec m})^j=0 \text{ and }p_j(l2\pi\sqrt{-1}+\vec m\cdot \vec c)\neq 0\text { for }j\in I.
\end{equation*}

\medskip\noindent
Let $R$ be the $\Q$-algebra generated by the coefficients of $p$, $\pi\sqrt{-1}$, $\vec c$, $\vec d$ 
and their inverses. Let $G$ be the multiplicative subgroup of $R^\times$ generated by $\vec d$.  
Choose $\phi_1,\dots,\phi_q$ ring homomorphisms from $R$ to $\bar\Q$ injective on $G$ as in Lemma 
\ref{lemme4laurent} and let $F$ be the compositum field of all their images. 

\medskip\noindent
Let $(l,\vec m)$ satisfy (\ref{MainEQ}) and choose $\nu$ in $\{1,\dots,q\}$ such that
$$\phi_\nu(p_0(l2\pi\sqrt{-1}+\sum\limits_{j=1}^{t'} m_j c_j))\neq 0.$$
The map $\phi_\nu$ transforms (\ref{MainEQ}) into
$$
 \sum_{j\in I}p_{jl}(\vec m)\prod_{i=1}^{t'}\phi_\nu(d_i^j)^{m_i}=0,
$$
where $p_{jl}(\vec X)$ is is a polynomial in $(1+t')$-variables such that 
$$p_{jl}(\vec m)=\phi_\nu(p_{j}(l2\pi\sqrt{-1}+\vec m\cdot \vec c)).$$ 
We may assume that no subsum is zero. Hence applying Theorem \ref{Laurent} and the remark after it, there is $T$ in $\N$ 
independent of $l$ such that $||\vec m||\leq T$. The proof finishes by noting that for each $\vec m$, there are finitely 
many $l$'s satisfying (\ref{MainEQ}).
 
\end{proof}

\begin{question}
Do the techniques used in the proof of \ref{T:main} carry over to the case of a system of two polynomial equations?  
\begin{align*}
 p_1(x,y,\ex(x),\ex(y)) &= 0 \\ p_2(x,y,\ex(x),\ex(y)) &= 0
\end{align*}
It is not clear to the authors how to show that there are infinitely many solutions to the above system, 
since the proof of that depends heavily on Hadamard's factorization theorem \cite{Hormander} for one single 
complex variable.  
\end{question}


\begin{thebibliography}{10}

\bibitem{Ax}
James Ax.
\newblock On {S}chanuel's conjectures.
\newblock {\em Ann. of Math. (2)}, 93:252--268, 1971.

\bibitem{vdDries-exp}
Lou van~den Dries.
\newblock Exponential rings, exponential polynomials and exponential functions.
\newblock {\em Pacific J. Math.}, 113(1):51--66, 1984.

\bibitem{vdDG}
Lou van~den Dries and Ayhan G{\"u}nayd{\i}n.
\newblock The fields of real and complex numbers with a small multiplicative
  group.
\newblock {\em Proc. London Math. Soc. (3)}, 93(1):43--81, 2006.

\bibitem{MP}
Lou van~den Dries and Ayhan G\"{u}nayd{\i}n.
\newblock Mann pairs.
\newblock {\em Trans. Amer. Math. Soc.}, 362(5):2393--2414, 2010.

\bibitem{dvornicich-Zannier}
Roberto Dvornicich and Umberto Zannier.
\newblock On sums of roots of unity.
\newblock {\em Monatsh. Math.}, 129(2):97--108, 2000.

\bibitem{Henson-Rubel}
C.~Ward Henson and Lee~A. Rubel.
\newblock Some applications of {N}evanlinna theory to mathematical logic:
  identities of exponential functions.
\newblock {\em Trans. Amer. Math. Soc.}, 282(1):1--32, 1984.

\bibitem{Hormander}
Lars H{\"o}rmander.
\newblock {\em An introduction to complex analysis in several variables},
  volume~7 of {\em North-Holland Mathematical Library}.
\newblock North-Holland Publishing Co., Amsterdam, third edition, 1990.

\bibitem{Laurent}
Michel Laurent.
\newblock Equations diophantiennes exponentielles.
\newblock {\em Invent. Math.}, 78:299--327, 1984.

\bibitem{LaurentII}
Michel Laurent.
\newblock \'{E}quations exponentielles-polyn\^omes et suites r\'ecurrentes
  lin\'eaires. {II}.
\newblock {\em J. Number Theory}, 31(1):24--53, 1989.

\bibitem{marker_exp}
David Marker.
\newblock A remark on {Z}ilber's pseudoexponentiation.
\newblock {\em J. Symbolic Logic}, 71(3):791--798, 2006.

\bibitem{Zilber-Pseudoexp}
Boris Zilber.
\newblock Pseudo-exponentiation on algebraically closed fields of
  characteristic zero.
\newblock {\em Ann. Pure Appl. Logic}, 132(1):67--95, 2005.

\bibitem{Bays-Zilber}
Boris Zilber and Martin Bays.
\newblock Covers of multiplicative groups of algebraically closed fields of
  arbitrary characteristic.
\newblock Available at the webpage \texttt{www.maths.ox.ac.uk/$\sim$zilber},
  2007.

\end{thebibliography}
\end{document}